\documentclass[]{amsart}

\usepackage{amssymb,amsthm}
\usepackage{enumitem}
\usepackage{mathtools,doi,url,hyperref}

\usepackage{crossreftools}
\usepackage[mathscr]{euscript}
\usepackage{xcolor}

\newtheorem{thm}{Theorem}[section]
\newtheorem{cor}[thm]{Corollary}
\newtheorem{lem}[thm]{Lemma}
\newtheorem{prop}[thm]{Proposition}
\theoremstyle{definition}

\theoremstyle{remark}
\newtheorem{exam}[thm]{Example}
\newtheorem{rem}[thm]{Remark}
\newcommand{\C}{\mathbb{C}}%Complex plane
\newcommand{\T}{\mathbb{T}}%Unit circle in field
\newcommand{\D}{\mathbb{D}}%Unit disk in Complex plane
\newcommand{\R}{\mathbb{R}}%Real numbers

\newcommand*\BMO{\textrm{BMO}}

\newcommand{\Herglotz}{K}
\newcommand{\essran}{\mathcal{R}}

\newcommand{\norm}[1]{\left\Vert#1\right\Vert}%norm
\newcommand{\abs}[1]{\left\vert#1\right\vert}%absolute value/modulus

\newcommand*\PV{\,\mathcal{PV}\hspace{-3pt}\int}

\DeclareMathOperator*\esssup{ess\,sup}

\DeclareMathOperator*\essinf{ess\,inf}

\newcommand\COCO[1]{[#1]^{\sim}}

\title[Sharp exponential integrability of conjugate functions]{Sharp exponential integrability of conjugate functions}



\author{David Norrbo and Jani Virtanen}

\address{Faculty of Science and Engineering, Åbo Akademi University, 20500 Åbo, Finland}
\email{david.norrbo@ntnu.no}

\address{University of Eastern Finland, Joensuu, Finland. University of Reading, Reading, UK. University of Helsinki, Helsinki, Finland}
\email{jani.virtanen@uef.fi, j.a.virtanen@reading.ac.uk, and
jani.virtanen@helsinki.fi}

\thanks{}

\begin{document}

\begin{abstract}
 We prove that if a real-valued function $f\in L^1$ on the complex unit circle has a gap of width at least $\pi$ in its essential range, then $\exp(\tilde f)$ is not integrable, where $\tilde f$ is the conjugate function. More generally, the exponential function can be replaced by any nonnegative convex function $Y$ satisfying
$$
\int_0^\infty Y(x) e^{-x}\, dx = \infty.
$$
We also establish a local version on arcs of the unit circle: Under a natural condition on the inverse images of the two sides of the gap, $Y(\tilde f)$ fails to be integrable on the arc; in particular, it suffices that one of these inverse images is not an interval modulo null sets. Furthermore, we obtain corresponding results for complex-valued functions and a $Y$-integrability version of Zygmund's result.
\end{abstract}

\maketitle

\section{Introduction}

The relation between a real-valued harmonic function $u$ on the unit disk $\D$ and its harmonic conjugate $v$ can be expressed in terms of the Cauchy--Riemann equations, which already appear in d'Alembert's 1752 work on fluid dynamics \cite[§58]{dAlembert-1752}. In the modern era, it was discovered that for boundary functions the conjugation operator has the following endpoint behavior: A function $f\in L^\infty(\T)$ on the unit circle $\T$ may have an unbounded conjugate function $\tilde f$ \cite[p.~363]{Fatou-1906}, but $\tilde f$ still belongs to the space $\BMO$ of functions of bounded mean oscillation \cite[Proposition~7.1]{Spanne-1966}. We recall that for $f\in \BMO(\T)$, its $\BMO$ seminorm is defined by
$$
    \norm{f}_*=\sup_{I\subset \T}\frac{1}{|I|}\int_I |f(t)-f_I|\,dt,\qquad
    f_I=\frac{1}{|I|}\int_I f(t)\,dt,
$$
where the supremum is taken over all arcs $I\subset\T$. The John--Nirenberg inequality \cite[(2),(3)]{John-1961} implies that there exist absolute constants $C,c>0$ such that
$$
    \bigl|\{t\in I: |f(t)-f_I|>\lambda\}\bigr|
    \leq C|I|\exp\left(-\frac{c\lambda}{\norm{f}_*}
    \right)
$$
for every $f\in\BMO(\T)$, every arc $I\subset\T$, and every $\lambda>0$. In particular, if $f$ is not constant, then
$$
    \exp\left(\gamma\frac{|f-f_{\T}|}{\norm{f}_*}
    \right)\in L^1(\T)
$$
for all sufficiently small $\gamma>0$. 

In a complementary direction, Zygmund~\cite{Zygmund-1929} studied how the behavior of $f$ controls the exponential integrability of its conjugate function $\tilde f$, and in particular, he proved the following classical result; see also \cite[Corollary~III.2.6]{Garnett2007}: Let $f\in L^\infty(\T)$ and suppose that $\norm{f}_\infty\leq \frac{\pi}{2}$. Then, for every $\lambda<1$,
\begin{equation}\label{e:Zygmund}
    \exp\bigl(\lambda|\tilde f|\bigr)\in L^1(\T).
\end{equation}
If, in addition, $f$ is continuous on $\T$, then~\eqref{e:Zygmund} holds for all $\lambda>0$.

Another classical result, due to Stein and Weiss \cite{Stein-1959}, has a direct consequence for the endpoint $\lambda=1$ in~\eqref{e:Zygmund}. They computed the distribution functions of the Hilbert transform $H\chi_E$ on $\mathbb R$ and of the conjugate function $\widetilde{\chi_E}$ on $\T$. Remarkably, these distribution functions depend only on the measure of $E$, and not on its geometry. In particular, if $E\subset\T$ satisfies $0<|E|<2\pi$, then, writing $\rho_E=2\chi_E-1$,
\begin{equation}\label{e:Stein-Weiss}
    \exp\bigl(|\tfrac{\pi}{2}\widetilde{\rho_E}|\bigr)
    \notin L^1(\T).
\end{equation}
Thus, since $\norm{\tfrac{\pi}{2}\rho_E}_\infty=\pi/2$, the restriction $\lambda<1$ in Zygmund's theorem cannot be replaced by $\lambda\leq1$. Notice also that $\tfrac{\pi}{2}\rho_E$ has a gap of width $\pi$ in its essential range.

A local version of~\eqref{e:Stein-Weiss} was obtained by Clancey \cite{Clancey-1984}, who showed that if $E\subset\T$ and $I\subset\T$ is an arc such that both $I\cap E$ and $I\setminus E$ have positive measure and neither is an arc modulo null sets, then
\begin{equation}\label{e:Clancey}
    \exp\left(\widetilde{\tfrac{\pi}{2}\rho_E}\right)\notin L^1(I).
\end{equation}
To prove this, Clancey uses the clever idea that exponential integrability allows one to construct analytic functions on the two sides of $\T$ whose boundary values match on $I$, which forces analytic continuation across $I$ and contradicts the geometry of $E$ on the arc.

Shargorodsky \cite{Shargorodsky-1994} subsequently considered the more general case and proved that for a real-valued function $g\in L^\infty(\T)$ with $\essinf_{\T}g>\tfrac{\pi}{2}$, if $I\subset\T$ is an arc such that $|I\cap E|>0$, $|I\setminus E|>0$, and at least one of $I\cap E$ and $I\setminus E$ is not an arc modulo null sets, then
$$
    \exp\left(\widetilde{g\rho_E}\right)\notin L^1(I).
$$
His argument reduces this statement, by means of Zygmund's theorem, to Clancey's result in~\eqref{e:Clancey}. We note, however, that it uses Clancey's result under the weaker assumption that only one of $I\cap E$ and $I\setminus E$ need fail to be an arc modulo null sets, whereas \cite{Clancey-1984} assumes this of both sets. This slight generalization can be proved directly and also follows from our local result (see Theorem~\ref{thm:local}). Shargorodsky used this result to obtain geometric necessary conditions for the Fredholmness of Toeplitz operators $T_a$ with bounded measurable symbols, assuming that $0$ is not a cluster value of $a$.

More recently, Gissy, Miihkinen, and Virtanen \cite{Gissy-2021} proved the sharp endpoint result in the global case. If $g\in L^\infty(\T)$ is real valued,
$$
    \essinf_{\T} g\geq \frac{\pi}{2},
$$
and $0<|E|<2\pi$, then there is a constant $C>0$ such that
$$
    \bigl|\{t\in\T:\widetilde{g\rho_E}(t)
    >\lambda\}\bigr|\geq Ce^{-\lambda}
$$
for every $\lambda\geq0$. In particular,
\begin{equation}\label{e:GMV}
    \exp\bigl(\widetilde{g\rho_E}\bigr)
    \notin L^1(\T).
\end{equation}
Equivalently, if a real-valued function $f\in L^\infty(\T)$ has a gap of width at least $\pi$ in its essential range, then $\exp(\tilde f)\notin L^1(\T)$. The proof is elementary and quite different from those of the preceding results, relying only on classical tools from complex and harmonic analysis. It converts the geometry of the range of the analytic extension of $g\rho_E$ into a subharmonic-function estimate on $\D$, and comparison with an explicit harmonic function on a vertical strip yields the sharp lower bound $Ce^{-\lambda}$ for the distribution function.

Two natural questions arising from these results are whether Shargorodsky's local theorem remains valid at the endpoint where the width of the gap is exactly $\pi$, and whether the assumption $f\in L^\infty(\T)$ can be relaxed to $f\in L^1(\T)$. In fact, the local endpoint question was raised in \cite[Remark~2]{Gissy-2021}, while the extension to unbounded functions was posed in \cite[Open Problem~2]{Gissy-2021}.

We answer both questions affirmatively. First, our global result, Theorem~\ref{thm:global}, extends the nonintegrability conclusion of~\cite{Gissy-2021} in~\eqref{e:GMV} to real-valued functions in $L^1(\T)$: If the essential range of $f$ has a gap of size at least $\pi$, then
\begin{equation}\label{e:global}
    \exp(\tilde f)\notin L^1(\T).
\end{equation}
As a consequence of the global theorem, we also extend the Hardy space result of \cite[Example~1]{Gissy-2021} from bounded to integrable real-valued symbols; see Corollary~\ref{cor:outer}.

Second, our main result, Theorem~\ref{thm:local}, establishes the corresponding conclusion on an arc $I\subset\T$, including the critical gap width $\pi$, and again allows $f\in L^1(\T)$. In fact, the gap condition is required only on $I$, and the geometric hypothesis can also be weakened. More precisely, after subtracting the midpoint of the gap, suppose that $|f|\geq\pi/2$ almost everywhere on $I$. It suffices that there are Lebesgue points $x_0,y_0\in I$, ordered according to the angular parametrization of $I$, such that
$$
    x_0<y_0,\qquad
    f(x_0)\geq\frac{\pi}{2},
    \qquad
    f(y_0)\leq-\frac{\pi}{2}.
$$
Then
$$
    \exp(\tilde f)\notin L^1(I).
$$
In particular, this condition holds if the inverse images in $I$ of the two sides of the gap have positive measure and at least one of them is not an arc modulo null sets.

Both the global and local results are proved in a more general form, with the exponential function replaced by a nonnegative convex function $Y\colon\R\to[0,\infty[$ satisfying
$$
    \int_0^\infty Y(x)e^{-x}\,dx=\infty.
$$
Under the respective hypotheses on $f$, we obtain
$$
    Y(\tilde f)\notin L^1(\T)
    \qquad\text{and}\qquad
    Y(\tilde f)\notin L^1(I).
$$
For the exponential function, the gap size $\pi$ is optimal by Zygmund's theorem. A consequence of Theorem~\ref{thm:finite} gives that the growth condition on $Y$ is also sharp. Indeed, in Example \ref{ex:IllustrativeExampleLocal}, we show that for the two-valued function  $f=\tfrac{\pi}{2}\rho_E$, with $E$ an arc, $Y(\tilde f)$ 
is integrable in a sufficiently small neighborhood of the endpoint where $\tilde f$ tends to $\infty$ if and only if
$$
    \int_0^\infty Y(x)e^{-x}\,dx<\infty.
$$
The results apply, for example, to a nonnegative convex function $Y$ satisfying $Y(x)\sim e^x/x$ as $x\to\infty$; see Example~\ref{ex:BetterThanExpLocal}. Thus, the conclusions are stronger than the failure of exponential integrability alone. Note also that the conclusions concern $Y(\tilde f)$ itself, rather than only $Y(|\tilde f|)$.

Finally, we give a partial answer to Open Problem~3 of~\cite{Gissy-2021}, which asks for a converse to Zygmund's theorem for complex-valued functions. More precisely, we obtain analogs of our global and local results for complex-valued functions in $L^1(\T)$, with the gap in the real essential range replaced by an open strip of width $\pi$ and arbitrary position and orientation; see Theorem~\ref{thm:complex}. In particular, if $f$ avoids such a strip almost everywhere on an arc $I$ and takes values on both sides of the strip on sets of positive measure, then
$$
    \exp(|\tilde f|)\notin L^1(I).
$$
Under an appropriate ordering condition on the inverse images of the two sides, the conclusion strengthens to
$$
    Y\bigl(\Re(e^{-i\theta}\tilde f)\bigr)\notin L^1(I),
$$
where $e^{i\theta}$ is the direction normal to the strip.

The global proof uses a subharmonic comparison principle from \cite{Burkholder-1977} to compare the analytic function with boundary values $f+i\tilde f$ with a conformal map onto a vertical strip. The local proof proceeds by a different mechanism from the earlier analytic-continuation arguments. We develop a new localization method, which we believe is also of independent interest, by cutting the arc at suitable Lebesgue points and periodically extending the function between them. We apply the global theorem to this periodic extension and then use a one-sided estimate for the Hilbert transform to transfer the resulting nonintegrability conclusion back to the original function on the arc. The proof of the sufficient condition for $Y$-integrability uses the same subharmonic comparison principle as the global proof, but with the roles of the conformal map and $f+i\tilde{f}$ reversed. Since $f+i\tilde f$ must then map into the strip, $f$ is assumed to be bounded.

The paper is organized as follows: After introducing the necessary notation and preliminary results in the next section, we find necessary conditions for $Y$-integrability on the whole circle in Section~\ref{sec:global}. In Section~\ref{sec:local}, we prove Theorem~\ref{thm:local} by reducing, under suitable conditions, the problem on an arc to the problem on the whole circle.
The complex-valued case is treated in Section~\ref{sec:complex}. In Section~\ref{sec:SuffCondAndPWgaps} we obtain sufficient conditions for $Y$-integrability, and introduce pointwise gaps \eqref{eq:pointwiseGap}, which are a more precise tool than gaps when characterizing integrability of the conjugate function.

\section{Preliminaries}
We begin by introducing some notation. For a set $E\subset\mathbb R$, $\abs{E}$ is the Lebesgue measure of $E$, while $m(E) = \abs{E}/(2\pi)$. If $E$ is bounded, then $L^p(E) , \, 0< p\leq \infty$ consists of real-valued functions on $E$ and the norm is given by 
\[
\norm{f}_{L^p(E)}^p = \int_E \abs{f}^p \, dm(t), \quad p<\infty,
\]
and $\norm{f}_{L^\infty(E)} = \esssup_{t\in E} \abs{f(t)}$. If $E$ is unbounded, the measure dm is replaced by $d\abs{\cdot} = 2\pi \, dm$. For complex valued functions, we write $L_{\C}^p, \, 0< p\leq \infty$.

Denote the unit circle by $\T=\{e^{i\theta}\in\C: \theta\in [0,2\pi[\}$, identified with $[0,2\pi[$, and let $f\in L^1(\T)$. The Poisson integral is given by the convolution formula
\[
P[f](re^{it}) = P_r * f(e^{it}) = \int_0^{2\pi} P_r(t-x) f(x) \, dm(x),
\]
where $re^{it}\in \D$ and $P_z,\, z\in\D$ is the Poisson kernel, \cite[p.~11]{Garnett2007}. Furthermore, the conjugate function $\tilde f$ is defined by
$$
	\tilde f (s) = \frac1{2\pi} \PV_0^{2\pi} f(t) \cot \frac{s-t}2\, dt
$$
for almost every $s\in[0,2\pi[$, where the integral is understood in the Cauchy principal value sense. Equivalently, $\tilde f$ is the boundary function of the harmonic conjugate of the Poisson integral $P[f]$ normalized to vanish at the origin.

With a density point of a Lebesgue measurable set $E\subset \R$, we mean an $x\in \mathbb R$ such that \\$\lim_{h\to 0^+}\abs{E\cap ]x-h,x+h[}/(2h) = 1$.  

Given a function $f\in L^1_{\textrm{loc}}(\mathbb R)$, a Lebesgue point is an $x\in \R$ for which for any family $\mathcal F_x$ of sets of bounded eccentricity at $x$, we have
\[
\lim_{m(U)\to 0, U\in \mathcal F_x } \frac{1}{\abs{U}} \int_U \abs{ f(t) - f(x) } \, dt  = 0.  
\]

Recall that almost every point in $E$ is a density point of $E$, and almost every point $x\in\mathbb R$ is a Lebesgue point of $f$, see for example Corollaries 3.1.6 and 3.1.7 in  \cite[p.~106--108]{Stein-2005}.

Before we give an overview of the sections, we introduce some notation. Given a measurable set $I$ in $\T$ (or $\R$), the essential range of a measurable function $f\colon I\to \C$ is given by 
$$
    \essran (f) = \bigcap\{  \overline{f(I\setminus N)} :  m(N) = 0\}.
$$
Let $g$ be a real valued measurable function on $\T$ (or $\R$), let $\delta>0$ and let $I\subset\T$ (or $I\subset\R$) be an arc (or interval).  We say that $g$ has a $\delta$-gap on $I$ if there is $c\in\R$ such that the sets
$$
    E_+=\{\zeta\in I: g(\zeta)\geq c+\tfrac\delta2\}
    \quad\text{ and }\quad
    E_-=\{\zeta\in I: g(\zeta)\leq c-\tfrac\delta2\}
$$
both have positive measure and satisfy $\abs{I\setminus(E_+\cup E_-)}=0$. The size, also called width, is given by 
\[
\sup\{\delta>0: g \text{ has a } \delta\text{-gap on } I\}.
\]

Traversing $I$ counterclockwise, we say that the gap is positively oriented if some density point of $E_+$ precedes some density point of $E_-$, and negatively oriented if some density point of $E_-$ precedes some density point of $E_+$. 
Notice that $g$ has a $\delta$-gap on $I$ if and only if there is a point $c\in\R$ such that
$$
    \essran(g|_I)\subset\,
    ]-\infty,c-\delta/2]
    \cup[c+\delta/2,\infty[,
$$
and $\essran(g|_I)$ contains a point of each of the two half-lines. Since the critical gap size for exponential integrability of the conjugate function is $\pi$, we define
\begin{equation}\label{eq:vertStrip}
    S=\left\{z\in\C:|\Re z|<\frac{\pi}{2}\right\}.
\end{equation}

Given a function $F$ analytic in $\D$, we denote the boundary function on $\T$ by $F^*$. We say that $F$ is an outer function if
$$
    F(z) = \alpha\exp\left(\int_\T \frac{\zeta+z}{\zeta-z}h(\zeta)\, dm(\zeta)\right), \qquad z\in\D,
$$
where $h\in L^1(\T)$ is real valued and $|\alpha|=1$. The Smirnov class is defined as
\[
N^+=\{f=g/h : g,h\in H^\infty,\ h \text{ outer function}\}.
\]
We denote by $L_0^{1,\infty}$ the small weak $L^1$ space which is a vector space of measurable functions such that 
\[
\lim_{t\to \infty} t m\{ \xi\in\T : \abs{f(\xi)} \geq t  \} = 0.
\]
Functions in $N^+$ whose boundary function is in $L_0^{1,\infty}$ constitute $H_0^{1,\infty}$. These classes will appear in Section \ref{sec:global}. The terminology increasing is used for a function satisfying $f(x)\leq f(y)$ if $x < y$, and the word strictly increasing means $f(x)< f(y)$ when $x<y$.

Finally, we mention that in Section \ref{sec:global} the circle and arc setting is used due to its natural connection to the complex unit disk, while Section \ref{sec:local} is better presented using the parametrized version of $\T$, that is, $[0,2\pi[$ due to the connection with Hilbert transform-type operators. It is also worth mentioning that given a set $J$, any real-valued  function $f\colon J \to \mathbb R$ can be written as 
\[
f = \abs{f} (2\chi_E -1 ),
\]
where $E := f^{-1}([0,\infty[)\subset J$. This representation has been used in, for example, \cite{Shargorodsky-1994} and \cite{Gissy-2021}. Hereafter, we denote $\rho_E = 2\chi_E-1$.

\section{Global results}\label{sec:global}

We use the following subharmonic form of Burkholder's comparison theorem. For $\Phi(w)=|w|^p$, it is \cite[(4.15)]{Burkholder-1977}, and for $\Phi(w)=\varphi(|w|)$ it is \cite[(4.18)]{Burkholder-1977}. Baernstein's proof of these results \cite[pp.~203--204]{Burkholder-1977} gives the full statement.

\begin{thm}\label{thm:burkholder}
Let $\Omega\subset\C$ be a domain and let $\Phi:\C\to[0,\infty[$ be subharmonic and continuous.  If $F$ and $G$ are analytic in $\D$ with
$$
    F(0)=G(0)\in\Omega,\qquad F(\D)\subset\Omega,
$$
and the nontangential boundary values of $G$ satisfy
$$
    G^*(\zeta)\in\C\setminus\Omega \qquad\text{for a.e. }\zeta\in\T,
$$
then
$$
    \sup_{0<r<1}\int_{\T}\Phi(F(r\zeta))\,dm(\zeta)
    \leq\sup_{0<r<1}\int_{\T}\Phi(G(r\zeta))\,dm(\zeta).
$$
\end{thm}

\begin{lem}\label{lem:strip}
Let $Y\colon\R\to[0,\infty[$ be convex. Set $\Phi_Y(z)=Y(\Im z)$. The function
$$
    F(z)=-i\log\frac{1+z}{1-z},\qquad z\in\D,
$$
where the logarithm is the branch on the right half-plane that is real on $]0,\infty[$, maps $\D$ conformally onto $S$. Further, for $\zeta\in\T\setminus\{\pm1\}$,
\begin{equation}\label{e:boundary formula}
    \Im F^*(\zeta)=\log\frac{|1-\zeta|}{|1+\zeta|}.
\end{equation}
Moreover, for every $b\in\R$, the function $F_b=F+ib$ satisfies
\begin{equation}\label{e:strip-divergence}
    \sup_{0<r<1}\int_{\T}\Phi_Y(F_b(r\zeta))\,dm(\zeta)\geq \frac{e^{-|b|}}{\pi}\int_{\R} Y(\omega)e^{-\abs{\omega}}\,d\omega.
\end{equation}
\end{lem}

\begin{proof}
The M\"obius transformation $z\mapsto(1+z)/(1-z)$ maps $\D$ conformally onto the right half-plane. The logarithm maps the right half-plane conformally onto the horizontal strip $\{w:|\Im w|<\pi/2\}$, and multiplication by $-i$ rotates that strip onto $S$.  Thus $F$ maps $\D$ conformally onto $S$.

A direct calculation gives
$$
    \Im F(z)=\log\frac{|1-z|}{|1+z|}.
$$
Taking nontangential limits gives~\eqref{e:boundary formula}.

Write $\zeta=e^{it}$. Then
$$
    \Im F^*(e^{it})=\log\left\lvert\tan\frac t2\right\rvert.
$$

Since $Y\geq0$, $\cosh(\omega-b)\leq e^{|b|}\cosh \omega$ and $\cosh \omega\leq e^{\abs{\omega}}$ for $\omega\in\R$, the change of variables $x=\log\tan(t/2)$ on $]0,\pi[$ together with the symmetry
$t\mapsto2\pi-t$ gives
\begin{align*}
    \int_{\T}Y\bigl(\Im F^*(\zeta)+b\bigr)\,dm(\zeta)
    &=\frac1\pi\int_{\R}\frac{Y(x+b)}{\cosh x}\,dx=\frac1\pi\int_{\R}\frac{Y(\omega)}{\cosh(\omega-b)}\,d\omega\\
    &\geq\frac{e^{-|b|}}{\pi}\int_{\R}\frac{Y(\omega)}{\cosh \omega}\,d\omega \geq \frac{e^{-|b|}}{\pi}\int_{\R} Y(\omega)e^{-\abs{\omega}}\,d\omega.
\end{align*}
Because $Y$ is continuous and nonnegative, Fatou's lemma implies
\begin{align*}
    \liminf_{r\to1^-} \int_{\T}\Phi_Y(F_b(r\zeta))\,dm(\zeta)
    &\geq \int_{\T}Y\bigl(\Im F^*(\zeta)+b\bigr)\,dm(\zeta),
\end{align*}
which proves \eqref{e:strip-divergence}.
\end{proof}

We use the Herglotz $A$-integral representation in \cite[Theorem~5.4]{GMR} to prove that integrability of either the positive or the negative part of a conjugate function implies its full integrability.

\begin{lem}\label{lem:one-sided}
Let $u\in L^1(\T)$ be real-valued. If $(\widetilde u)^+\in L^1(\T)$ or $(\widetilde u)^-\in L^1(\T)$, then $\widetilde u\in L^1(\T)$.
\end{lem}

\begin{proof}
Write $v=\widetilde u$ and define
$$
    v_A=\max\{-A,\min\{v,A\}\},
$$
where $A>0$. For completeness, we first derive the following sharpened form of Kolmogorov's weak-type estimate (which is stated in Section~5.1 of \cite{GMR}, although the reference given there is to the classical estimate):
\begin{equation}\label{e:one-sided}
	\lim_{A\to \infty} A\,m\{\xi\in \T : |v(\xi)|>A\} = 0.
\end{equation}
Indeed, let $C>0$ be a constant such that Kolmogorov's weak-type $(1,1)$ estimate
$$
    m\bigl\{|\widetilde w|>\lambda\bigr\}\leq \frac{C\|w\|_{L^1(\mathbb T)}}{\lambda}
$$
holds for every real-valued $w\in L^1(\mathbb T)$ and every $\lambda>0$. Given $\epsilon>0$, choose a real trigonometric polynomial $p$ such that $\|u-p\|_{L^1}<\epsilon$.  Since $\widetilde p$ is bounded, for $A>2\|\widetilde p\|_\infty$,
$$
    \{|v|>A\}\subset\left\{|\widetilde{u-p}|>\frac A2\right\},
$$
and so
$$
    A\,m\{|v|>A\}\leq 2C\|u-p\|_{L^1}<2C\epsilon,
$$
which proves \eqref{e:one-sided}.

Since \cite[Theorem~5.4]{GMR} requires the real part of the analytic function to vanish at the origin, set
$$
    u_0=u-\int_{\T}u\,dm
$$
and define the Herglotz integral
$$
    \Herglotz_{u_0}(z)=\int_{\T}\frac{\xi+z}{\xi-z}\,u_0(\xi)\,dm(\xi), \qquad z\in\D.
$$
Write $\Herglotz_{u_0}=U+iV$, where $U=P[u_0]$ and $V$ is the harmonic conjugate of $U$ normalized by $V(0)=0$. Since $m$ is normalized, $U(0)=\Re \Herglotz_{u_0}(0)=\int_{\T}u_0\,dm=0$. We still need to verify that $\Herglotz_{u_0}\in H^{1,\infty}_0$. Since $U=P[u_0]\in h^1$, Kolmogorov's theorem for conjugate harmonic functions implies that $V\in h^p$ for every $0<p<1$. Moreover, $U\in h^p$ for every $0<p<1$, and
$$
    |\Herglotz_{u_0}|^p\leq (|U|+|V|)^p\leq |U|^p+|V|^p.
$$
Hence $\Herglotz_{u_0}=U+iV\in H^p$ for every $0<p<1$, and therefore $\Herglotz_{u_0}\in N^+$.

Notice that $\Herglotz_{u_0}^*=u_0+iv$ a.e.\ on $\T$, where $v=\widetilde u=\widetilde{u_0}$. To show that
$\Herglotz_{u_0}^*\in L^{1,\infty}_0$, observe that
$$
    \{|\Herglotz_{u_0}^*|>A\}\subset\{|u_0|>A/2\}\cup\{|v|>A/2\}.
$$
Since $u_0\in L^1(\T)$,
$$
    A\,m\{|u_0|>A/2\}\leq2\int_{\{|u_0|>A/2\}}|u_0|\,dm\to 0.
$$
Further, by \eqref{e:one-sided},
$$
    A\,m\{|v|>A/2\}=2\left(\frac A2\right)m\{|v|>A/2\}\to 0.
$$
Thus, $A\,m\{|\Herglotz_{u_0}^*|>A\}\to 0$ and so $\Herglotz_{u_0}^*\in L^{1,\infty}_0$. Since $\Herglotz_{u_0}\in N^+$, it follows that
$\Herglotz_{u_0}\in H^{1,\infty}_0$. Together with $U(0)=0$, this verifies all the hypotheses of \cite[Theorem~5.4]{GMR}, which gives
$$
    \Herglotz_{u_0}(z)=\lim_{A\to\infty} i\int_{\T}\frac{\xi+z}{\xi-z}\,v_A(\xi)\,dm(\xi)
$$
uniformly on compact subsets of $\D$. Now
\begin{align*}
    0 &= \Herglotz_{u_0}(0) = \lim_{A\to \infty} i\int_{\T}v_A\,dm\\
	&=\lim_{A\to \infty} i \left( \int_{\T}\min\{v^+,A\}\,dm-\int_{\T}\min\{v^-,A\}\,dm\right).
\end{align*}
Suppose first that $v^+\in L^1(\T)$. Since $\min\{v^+,A\}\to v^+$ as $A\to\infty$, monotone
convergence shows that
$$
    \int_{\T}\min\{v^+,A\}\,dm\to\int_{\T}v^+\,dm<\infty
$$
and so
$$
    \lim_{A\to\infty}\int_{\T}\min\{v^-,A\}\,dm=\int_{\T}v^+\,dm<\infty.
$$
Since $\min\{v^-,A\}\to v^-$, another application of monotone convergence implies that $v^-\in L^1(\T)$. Thus, $v\in L^1(\T)$. The argument with the signs reversed proves the other case.
\end{proof}

We use the following elementary consequence of Jensen's inequality and Fatou's lemma in the proof of Theorem~\ref{thm:global}.

\begin{lem}\label{lem:supEqualsBdd}
Let $h\in L^1(\T)$, and let 
\[
u(re^{it}) = P_r*h(e^{it}) , \, re^{it}\in\D.
\]
If $Y\colon \mathbb R \to [0,\infty[$ is convex, then
\[
\sup_{0<r<1} \int_{\T} Y\bigl(u(r\zeta)\bigr)\,dm(\zeta) = \lim_{r\to 1^-} \int_{\T} Y\bigl(u(r\zeta)\bigr)\,dm(\zeta)
    =\int_{\T}Y(h(\zeta))\,dm(\zeta).
\]
\end{lem}
\begin{proof}
Jensen's inequality for the convex function $Y$ shows that
$$
    Y\bigl(u(re^{it})\bigr) =Y\bigl(P_r*h(e^{it})\bigr) \leq P_r*(Y\circ h)(e^{it}),
$$
and so Young's convolution inequality and Fatou's lemma gives
\begin{equation*}
    \sup_{0<r<1} \int_{\T} Y(u(re^{it}))\,dm(\zeta)
    \le\int_{\T}Y(h(\zeta))\,dm(\zeta) \leq \liminf_{r\to 1} \int_{\T} Y(u(re^{it}))\,dm(\zeta).
\end{equation*}
\end{proof}

\begin{thm}\label{thm:global}
Let $Y\colon\R\to[0,\infty[$ be convex. If $f\in L^1(\T)$ is real valued and satisfies the condition $0<|f^{-1}([\pi/2,\infty[)|<2\pi$, and if $f^{-1}(]-\pi/2,\pi/2[)$ has measure zero, then
\begin{equation}\label{e:thm:global}
        \int_{\T}Y(\tilde f)\,dm \gtrsim_f \int_{\R}Y(\omega)e^{-\abs{\omega}}\,d\omega.
\end{equation}
In particular, if
\[
\int_{\R} Y(\omega)e^{-\abs{\omega}}\,d\omega = \infty,
\]
then
\[
Y(\tilde f) \notin L^1(\T).
\]
\end{thm}

\begin{proof}
The case where $Y$ is constant is immediate. Otherwise, since $Y$ is nonconstant and convex, it is unbounded on $[0,\infty[$ or $]-\infty,0]$. The statement is invariant under $f\mapsto-f$ and $Y(x)\mapsto Y(-x)$, so we may assume that $Y$ is unbounded on $[0,\infty[$.

Set $\Phi_Y(z)=Y(\Im z)$. Since $Y\geq 0$ is convex, $\Phi_Y$ is nonnegative and subharmonic on $\C$. Define
$$
    G(z)=\int_{\T}\frac{\xi+z}{\xi-z}\,f(\xi)\,dm(\xi),\qquad z\in\D.
$$
Then $\Re G=P[f]$, $\Im G(0)=0$, and $G^*=f+i\tilde f$ a.e. on $\T$. Thus,
\begin{equation}\label{e:leaves}
    G^{*}(\zeta)\in\C\setminus S\qquad\text{for a.e.\ }\zeta\in\T ,
\end{equation}
where $S$ is defined in \eqref{eq:vertStrip}. Choose Lebesgue points $\zeta_+\in f^{-1}([\pi/2,\infty[)$ and $\zeta_-\in f^{-1}(]-\infty,-\pi/2])$ of $f$ such that $f(\zeta_+)\geq\pi/2$ and $f(\zeta_-)\leq-\pi/2$. Since $P[f](r\zeta_\pm)\to f(\zeta_\pm)$ as $r\to1^-$, the harmonic function $\Re G=P[f]$ has a positive and a negative value in $\D$. By continuity, there is $z_0\in\D$ such that $\Re G(z_0)=0$.

Let $\sigma$ be an automorphism of $\D$ with $\sigma(0)=z_0$, and put $G_\sigma=G\circ\sigma$. By the same Kolmogorov argument used in Lemma \ref{lem:one-sided}, $G\in H^p$ for every $0<p<1$, and hence $G\in N^+$, and so also $G_\sigma\in N^+$ because $N^+$ is invariant under disk automorphisms. Since $G_\sigma^*=G^*\circ\sigma$ a.e.\ on $\T$, by \eqref{e:leaves},
$$
    G_\sigma^*(\zeta)\in\C\setminus S\qquad\text{for a.e. }\zeta\in\T.
$$

We may suppose that
\begin{equation}\label{e:contradiction-assumption}
    \int_{\T}Y(\tilde f)\,dm<\infty.
\end{equation}
Since $Y$ is unbounded above on $[0,\infty[$, we can choose $x_1>0$ such that $Y(x_1)>Y(0)$, and set
$$
    a=\frac{Y(x_1)-Y(0)}{x_1}>0.
$$
By convexity, $Y(x)\geq Y(x_1)+a(x-x_1)$ for $x\geq x_1$, and so, if $x\geq x_1$, then
$$
    x\leq a^{-1}\bigl(Y(x)-Y(x_1)+ax_1\bigr) \leq a^{-1}\bigl(Y(x)+ax_1\bigr).
$$
For $0\leq x\leq x_1$, we have $x^+\leq x_1\leq x_1(1+Y(x))$ since $Y\geq0$, while $x^+=0$ for
$x<0$. Thus, there is a constant $C>0$ such that
$$
    x^+\leq C\bigl(1+Y(x)\bigr)
$$
for $x\in\R$. It follows from \eqref{e:contradiction-assumption} that $(\tilde f)^+\in L^1(\T)$, and so, by Lemma~\ref{lem:one-sided}, we have $\tilde f\in L^1(\T)$.

Let $h=\Im G_\sigma^*=\tilde f\circ\sigma$. Since $f,\tilde f\in L^1(\T)$ and $|(\sigma^{-1})'|$ is bounded on $\T$, a change of variables shows that $f\circ\sigma,h\in L^1(\T)$. Hence
$$
    G_\sigma^*=(f+i\tilde f)\circ\sigma\in L^1(\T).
$$
Since $G_\sigma\in N^+$, Smirnov's theorem implies that $G_\sigma\in H^1$. Therefore
$$
    \Im G_\sigma(re^{it})=P_r*h(e^{it}), \qquad 0<r<1.
$$
Moreover, since $|(\sigma^{-1})'|$ is bounded on $\T$,
\begin{equation}\label{eq:ShiftingZero}
    \int_{\T}Y(h(\zeta))\,dm(\zeta) =\int_{\T} Y(\tilde f(\eta)) |(\sigma^{-1})'(\eta)|\,dm(\eta)\leq \|(\sigma^{-1})'\|_{\infty} \int_{\T} Y(\tilde f(\eta)) \,dm(\eta) <\infty
\end{equation}
by \eqref{e:contradiction-assumption}. Lemma \ref{lem:supEqualsBdd} gives
\begin{equation}\label{e1:global}
    \sup_{0<r<1} \int_{\mathbb T} \Phi_Y\bigl(G_\sigma(r\zeta)\bigr)\,dm(\zeta)
    \le\int_{\mathbb T}Y(h)\,dm.
\end{equation}

Since $\Re G_\sigma(0)=0$, we can write $G_\sigma(0)=ib$ for some $b\in\mathbb R$. Since $F(0)=0$, the function $F_b=F+ib$ satisfies $F_b(0)=G_\sigma(0)$ and $F_b(\D)=S$. By Lemma~\ref{lem:strip} and Theorem~\ref{thm:burkholder},
$$
    \frac{e^{-|b|}}{\pi}\int_{\R} Y(\omega)e^{-\abs{\omega}}\,d\omega\leq\sup_{0<r<1}\int_{\T}\Phi_Y\bigl(F_b(r\zeta)\bigr)\,dm(\zeta)
    \leq\sup_{0<r<1}\int_{\T}\Phi_Y\bigl(G_\sigma(r\zeta)\bigr)\,dm(\zeta),
$$
which, together with \eqref{eq:ShiftingZero} and \eqref{e1:global}, shows that
\begin{equation}\label{e:thm:global-precise}
    \int_\T Y(\tilde f)\, dm
    \ge \frac{e^{-|b|}}{\pi\|(\sigma^{-1})'\|_\infty}
    \int_\R Y(\omega)e^{-|\omega|}\, d\omega.
\end{equation}
\end{proof}

As a consequence of Theorem \ref{thm:global}, we provide an improved result of non-exponential integrability of the harmonic conjugate of a function whose essential range has a gap of size at least $\pi$.

\begin{exam}\label{ex:BetterThanExp}

Let $Y(x)=(e^x/(1+x)-1)\chi_{[0,\infty[}(x)$. If $E\subset\T$ is measurable with $0<|E|<2\pi$ and $f\in L^1(\T)$ satisfies $f\geq\pi/2$ a.e., then
$$
    Y\bigl(\widetilde{\rho_E f}\bigr)\notin L^1(\T).
$$

The only thing we need to verify is that $Y$ is convex, since $\int_0^\infty (1+x)^{-1}  \, dx =\infty$.
Differentiating, we obtain for $x\neq 0$,
\[
Y'(x) = \begin{cases}
\frac{  xe^x }{(1+x)^2}, & x>0; \\
0, &x<0;
\end{cases} \quad \text{ and } \quad Y''(x) = \begin{cases}
\frac{  (x^2+1)e^x }{(1+x)^3}, & x>0; \\
0, &x<0.
\end{cases}  
\]
Since both $Y$ and $Y'$ are continuous on $\mathbb R$, $Y$ is increasing and convex on $\mathbb R$.

\end{exam}

For the subclass $f\in L^\infty$, Example \ref{ex:BetterThanExp} follows from the distribution estimate in \cite[Theorem 1]{Gissy-2021}. A function $Y\in C^\infty(\mathbb R)$ with similar growth is given by $Y(x) = e^x(\pi/2 - \arctan(x/2))$, but since proving it is convex is quite tedious and relies on basic calculus, we leave out the details. This function can in fact be analytically extended to $\C\setminus \bigcup \pm i[2,\infty[$.

The global theorem has the following Hardy-space consequence, which extends \cite[Example~1]{Gissy-2021} from bounded to integrable symbols.

\begin{cor}\label{cor:outer}
Let $u\in L^1(\T)$ be real valued. Suppose that there are real numbers
$a<b$ such that
$$
    u\in]-\infty,a]\cup[b,\infty[\qquad\text{a.e. on }\T,
$$
and that the two sets $\{\xi\in\T : u(\xi)\leq a\}$ and $\{\xi\in\T : u(\xi)\geq b\}$ have positive measure. Let $\delta=b-a$ and define
$$
    \Phi_u(z)=\exp\left\{
    \frac{i}{2}\int_{\T}\frac{\zeta+z}{\zeta-z}\,
    u(\zeta)\,dm(\zeta)\right\},
    \qquad z\in\D.
$$
Then
\begin{equation}\label{e:outer}
    \Phi_u,\ \Phi_u^{-1}\notin H^p \quad\text{whenever}\
    p\geq\frac{2\pi}{\delta}.
\end{equation}
If, in addition, $\tilde u\in L^1(\T)$, then $\Phi_u$ and $\Phi_u^{-1}$ are outer functions.
\end{cor}

\begin{proof} 
Define $v=\tfrac\pi\delta(u-\tfrac12(a+b))$. Then 
$$
|v|\geq \frac{\pi}{2}\quad\text{a.e.} \quad \text{ and }  \quad 0<\abs{ \Bigl\{\zeta\in\T : v(\zeta)\geq \frac{\pi}{2} \Bigr\} } <2\pi.
$$

Applying Theorem~\ref{thm:global} with $Y(x)=e^x$ implies that $\exp(\widetilde v)\notin L^1(\T)$ and $\exp(-\widetilde v)\notin L^1(\T)$, the latter by using $-v$ instead of $v$. Thus, 
$$
    |\Phi_u^*|^{2\pi/\delta}=\exp\left(-\frac{\pi}{\delta}\widetilde u\right) \notin L^1(\T)\quad and\quad
    |(\Phi_u^{-1})^*|^{2\pi/\delta}
    =
    \exp\left(\frac{\pi}{\delta}\widetilde u\right)
    \notin L^1(\T),
$$
and so $\Phi_u,\Phi_u^{-1}\notin H^{2\pi/\delta}$. Since
$H^p\subset H^{q}$ whenever $p\geq q$, we obtain~\eqref{e:outer}.

Finally, suppose that $\widetilde u\in L^1(\T)$ and write
$$
    \Herglotz_u(z)=\int_{\T}\frac{\zeta+z}{\zeta-z}u(\zeta)\,dm(\zeta)
$$
for the Herglotz integral as in the proof of Lemma \ref{lem:one-sided}. Then
$$
     \Herglotz_{-\widetilde u/2}=\frac{i}{2}\Herglotz_u-\frac{iu_0}{2}, \quad \text{where}\  u_0=\int_{\T}u\,dm,
$$
and so
$$
    \Phi_u=e^{iu_0/2}
    \exp\left\{\int_{\T}\frac{\zeta+z}{\zeta-z}
    \left(-\frac12\widetilde u(\zeta)\right)\,dm(\zeta)\right\}
$$
is an outer function.  The same argument, with $\widetilde u/2$ in place of $-\widetilde u/2$, shows that $\Phi_u^{-1}$ is outer.
\end{proof}

In the following remark, we compare the condition on $Y$ in Theorem~\ref{thm:global} with the condition that follows from the distribution estimate in \cite[Theorem~1]{Gissy-2021}, which is sharp by \cite[Theorem~(b)]{Stein-1959}.

\begin{rem}\label{rem:CompareWithEarlierResult}

In \cite[Theorem 1]{Gissy-2021} a precise lower bound on the growth rate was calculated for the distribution 
\[
D_{\tilde{f}}(x) = m\{ t\in\T : \tilde{f}(t) > x \}.
\]
The authors proved that given $f\in L^\infty $ with $f\geq \pi/2$ a.e. and $0<\abs{E}<2\pi$, it holds that $D_{ \widetilde{\rho_E f}  }(x) \gtrsim e^{-x}, \, x>0$. Assume that $Y\in C^1([0,\infty[)$ is nonnegative, strictly increasing and convex on $]0,\infty[$. It follows that  
\[
\int_{\T} Y\bigl((\widetilde{\rho_E f})^+(t)\bigr) - Y(0) \, dm(t)  = \int_0^\infty D_{ \widetilde{\rho_E f}  }(x) Y'(x) \, dx \gtrsim  \int_0^\infty e^{-x} Y'(x) \, dx. 
\]
Hence, using the precise uniform lower bound of the distributional growth, we obtain non-$Y$ integrability of the conjugate function, whenever $ \int_0^\infty e^{-x} Y'(x) \, dx = \infty$. By integration by parts, we obtain
\[
\int_0^\infty e^{-x} Y'(x) \, dx \asymp \int_0^\infty e^{-x} (Y(x) - Y(0)) \, dx, 
\]
proving that for this class of $Y$, Theorem \ref{thm:global} describes the non-integrability as good as the uniform lower bound for the distribution function would. This could be explained by the fact that $g:=\tfrac\pi2\rho_E$, where $E$ is an arc, is a function proving that the lower bound of the distributional growth cannot be larger than $e^{-x}$. The harmonic conjugate of $g$ has logarithmic singularities, which can be compared to the function $F$ in Lemma \ref{lem:strip}, which is the function we used for the lower bound, see the end of the proof to Theorem~\ref{thm:global}.

The advantage of Theorem \ref{thm:global} compared to \cite[Theorem 1]{Gissy-2021} is that we allow $f\in L^1$, not only $f\in L^\infty$. The proof of \cite[Theorem 1]{Gissy-2021} cannot trivially be generalized to include unbounded functions, since it relies on the relation between $\norm{\cdot}_\infty$ and point-wise estimates, which is unique for $\norm{\cdot}_\infty$.

\end{rem}

\section{Local results}\label{sec:local}

The Hilbert transform of $f\in L^1(\mathbb R)$ is defined as
\[
Hf(x) = \frac{1}{\pi} \lim_{\epsilon\to 0} \int_{\mathbb R\setminus ]x-\epsilon,x+\epsilon[}  f(t) \frac{1}{x-t} \, dt = -\frac{1}{\pi} \lim_{\epsilon\to 0} \int_{\mathbb R\setminus ]-\epsilon,\epsilon[}  f(t+x) \, \frac{dt}{t}, 
\]
where the limit exists for almost every $x\in \mathbb R$. For $0<\epsilon < \delta$ define the bounded operator $L^1(\R) \to \R$
\[
H_{\epsilon}^{(\delta)}f = \int_{]-\delta, \delta[\setminus [-\epsilon,\epsilon] } f(t)  \, \frac{ dt}{t} = \int_{\epsilon}^\delta  \big( f(t) - f(-t) \big) \, \frac{dt}{t},
\]
and 
\[
Mf(s) = \frac{1}{s} \int_0^{s} \big( f(t) - f(-t) \big) \,dt, \quad  s\neq 0.
\]
Notice that $Mf$ is odd. Indeed, substituting $t\mapsto-t$, we have
\[
Mf(-s) = \frac{1}{-s} \int_0^{-s} \big( f(t) - f(-t) \big) \,dt \stackrel{ t\mapsto -t }{=} \frac{1}{s} \int_0^{s} \big( f(-t) - f(t) \big) \,dt = - Mf(s).
\]
Therefore,
\[
H_{\epsilon}^{(\delta)}Mf =  2\int_{\epsilon}^{\delta} Mf(t) \, \frac{ dt}{t} .
\]
By Fubini's theorem, we have
\begin{equation}\label{eq:eqForDeltaHilbertTransform}
\frac{1}{2}H_{\epsilon}^{(\delta)}Mf  = H_{\epsilon}^{(\delta)}f + Mf(\epsilon) - Mf(\delta).
\end{equation}
Indeed,
\begin{align*}
\frac{1}{2}H_{\epsilon}^{(\delta)}Mf &=  \int_{\epsilon}^{\delta} \frac{1}{t} \int_0^t \big( f(s) - f(-s) \big) \,ds \, \frac{dt}{t}  = \int_0^{\delta}  \int_0^{\delta}  \chi_{\epsilon<t<\delta} \chi_{0<s<t}  \big( f(s) - f(-s) \big) \,  \frac{dt}{t^2}   \,  ds \\
&=  \int_0^{\delta}  \int_{\max\{\epsilon,s\}}^{\delta}  \,  \frac{dt}{t^2} \, \big( f(s) - f(-s) \big) \,  ds =  \int_0^{\delta} \bigg(\frac{1}{\max\{\epsilon,s\}} - \frac{1}{\delta}\bigg)   \big( f(s) - f(-s) \big) \,  ds \\
&=\int_\epsilon^{\delta} \frac{1}{\max\{\epsilon,s\}}   \big( f(s) - f(-s) \big) \,  ds  + \int_0^{\epsilon} \frac{1}{\max\{\epsilon,s\}}   \big( f(s) - f(-s) \big) \,  ds  -Mf(\delta) \\
&=H_{\epsilon}^{(\delta)} f + Mf(\epsilon) - Mf(\delta).
\end{align*}

\bigskip

The next lemma shows the ($T$-)integrability of the Hilbert-transform of a function $f$ is not affected by scaling and translations. We will apply it to $T = Y$, which is a generalization of the exponential function.

\begin{lem}\label{HilbertTransformAndTranslations}
Let $I\subset \mathbb R$, $\psi(x) = ax+b, a>0,\, b\in\mathbb R$, let $T\colon \mathbb R \to \mathbb R$ be a measurable function, and let $f\in\mathbb L^1(\mathbb R)$. Then
\[
\int_I T(Hf(x))  \, dx  = a \int_{\psi^{-1}(I)} T(  H(f\circ \psi)(x)  )  \, dx.
\]

\end{lem}
\begin{proof}
The inverse of $\psi$ is given by $\psi^{-1}(x) = (x-b)/a$. It follows that
\[
\int_I T(Hf(x))  \, dx  \stackrel{x\mapsto \psi(x) }{=} a \int_{\psi^{-1}(I)} T(Hf(\psi(x)))  \, dx. 
\]
The result now follows from
\begin{align*}
Hf(\psi(x)) &= \frac{1}{\pi}\lim_{\epsilon\to 0}\int_{\mathbb R\setminus ]\psi(x)-\epsilon, \psi(x)+\epsilon[} f(t) \frac{1}{ \psi(x) - t }  \, dt \stackrel{t\mapsto \psi(t)}{=}  \frac{1}{\pi} \lim_{\epsilon\to 0}\int_{\mathbb R\setminus ]x-\epsilon/a, x+\epsilon/a[} f(\psi(t)) \frac{1}{x - t }  \, dt \\
&=  H(f\circ \psi)(x).
\end{align*}
\end{proof}

\begin{lem}\label{HilbertTransformAndConjugate}
Let $f\in L^1([0,2\pi[)$ and 
\[
f_e(x)=\sum_{k=-1}^1 \breve f(x+2\pi k)\chi_{[0,2\pi[}(x+2\pi k), \quad x\in \mathbb R,
\]
where $\breve f$ is the periodic extension of $f$ to $\R$. Then
\begin{equation}\label{HilbSimilarToHarmonicConj}
 \sup_{f\in B_{L^1([0,2\pi[)}} \sup_{t\in]0,2\pi[} \abs{    \tilde{f}(t) - H f_e(t)       } <\infty
\end{equation}

\end{lem}
\begin{proof}
Notice that
$$
    Hf_e(x) = \frac1\pi\PV_0^{2\pi} f(t)
    \left(\frac1{x-t-2\pi} + \frac1{x-t}+\frac1{x-t+2\pi}\right)dt,
$$
so
$$
   \tilde f(x)  - Hf_e(x) =\frac1\pi\PV_0^{2\pi} f(t)K(x-t)\,dt,
$$
where
$$
    K(t)= \frac12\cot\frac{t}2 -\frac1{t-2\pi}-\frac1t -\frac1{t+2\pi}
$$
has removable singularities and extends continuously to $[-2\pi,2\pi]$. Thus, 
$$
    \|\tilde f - Hf_e\|_{L^\infty([0,2\pi[)}
    \le2\|K\|_{L^\infty([-2\pi,2\pi[)}\|f\|_{L^1([0,2\pi[)},
$$
which completes the proof.
\end{proof}

Our strategy to deal with integration on a closed interval $[a,b]\subset [0,2\pi[$ is to map it by a linear transform to a problem on $[0,2\pi[$. However, calculating the harmonic conjugate, values outside of $[a,b]$ are used, therefore, it is clear that applying a linear transform and forcing periodicity will change the problem.
To this end, the first step is to consider $Y$-integrability of $\tilde f$ instead of $|\tilde f|$. This allows us to distinguish between singularities where the Hilbert transform is unbounded from above, and those where it is unbounded from below. We will then consider an increasing $Y$, so that $Y(Hf)$ can only be unbounded whenever $Hf$ is unbounded from above. Therefore, it is easier to choose how the cuts of $I$ should be made in order to allow us to glue the endpoints together (force periodicity) without changing the problem. 
The following example shows that even if we can make the cuts at Lebesgue points, such that when we glue the endpoints together, that point is still a Lebesgue point, we might introduce a new singularity. 
\begin{exam}\label{ex:NaiveGLue}
Let $f(x)$ be an even function defined on $\mathbb R$ as 
\[
\begin{cases} 
(\log \frac{1}{\abs{x}})^{-1}& x\in ]-1/e,1/e[\setminus \{0\};\\ 
0& x=0 \text{ or } \abs{x}>2;
\end{cases}
\]
and such that it is $C^1$ on an open neighborhood $N$ of $[-2,2]\setminus ]-1/e,1/e[$. Recall, that on the set where the function satisfies the Dini condition, the Hilbert transform will be bounded. 
Since $f$ is continuous on $\mathbb R$, every point is a Lebesgue point. 
We will now prove that the Hilbert transform is bounded on $\mathbb R$. Due to the observations above, it is sufficient to examine a neighborhood of $0$. Since $f$ is even in a neighborhood of zero, $Hf$ is bounded at $0$. By symmetry we can assume $0<x<\frac{1}{2e}$. Since $f$ is even, we have

\begin{align*}
    \pi H(f\chi_{]-1/e,1/e[})(x) &= \int_0^{1/e} \frac{f(t)dt}{x+t} + \PV_0^{1/e}\frac{f(t)dt}{x-t}\\
    &=2x\PV_0^{1/e} \frac{f(t)dt}{x^2-t^2}=2\PV_0^{1/(ex)}\frac{f(tx)dt}{1-t^2}.
\end{align*}

Since $f$ is bounded, we have for the tail
\[
\sup_{x\in]0,\frac{1}{2e}[}\abs{\int_2^{\frac{1}{e x}}  \frac{ f(tx)  }{1-t^2} \, dt} <\infty,
\]

and using $1/(1-t^2)=1/(2(1-t))+1/(2(1+t))$,
\begin{align*}
    \PV_0^2\frac{f(tx)}{1-t^2}\,dt
    &=\frac12\PV_0^2\frac{f(tx)}{1-t}\,dt+O(1)\\
    &=\frac12\int_0^1\frac{f(tx)-f((2-t)x)}{1-t}\,dt+O(1).
\end{align*}
To see that the integrand is bounded, we have for $0<x<1/(2e)$ and $0<t<1$,
$$
    \log\frac1{tx}\geq \log\frac{2e}{t}
    \quad\text{and}\quad
    \log\frac1{(2-t)x}\geq 1,
$$
so
$$
    \left|\frac{\log\frac1{2-t}-\log\frac1t}
    {(1-t)\log\frac1{tx}\log\frac1{(2-t)x}}\right|
    \le\frac{\log\Big(1+\frac{2(1-t)}{t} \Big) }{(1-t)\log\frac{2e}{t}}.
$$

Therefore, $H(f\chi_{]-1/e,1/e[})$ is bounded in a neighborhood of $0$, and hence, $Hf$ is bounded on $\mathbb R$. A slightly more careful analysis shows that after we split the integration and do a substitution to eliminate the principal value, we can apply dominated convergence, to obtain that $Hf$ is continuous at $0$, and due to the definition it is continuous on $\R$.
However, assume we make a cut at the Lebesgue points $0$ and $3$, and let $f_e$ be the function that is supported and $3$-periodic on $[-3,6[$, and equals $f$ on the interval $[0,3[$. The point $0$ would still be a Lebesgue point of $f_e$, due to the function still being continuous. Despite that, the Hilbert transform has now a singularity at the point, which follows from (notice $f_e$ is constant zero in a left neighborhood of $0$)
\[
\lim_{\epsilon\to 0^+}\int_{\epsilon}^{\frac{1}{e}} \frac{dt}{t\log(1/t)}
    =\infty.
\]
\end{exam}

The following lemma gives a situation in which $f$ can be replaced by a finite periodic extension of $f|_{[x_0,y_0[}$ without introducing a new blow-up point for its Hilbert transform on $]x_0,y_0[$. This allows us to glue together the endpoints after a suitable cut and to retrieve non-$Y$ integrability of the conjugate function from the Hilbert transform.

\begin{lem}\label{lem:mainLem}
Let $f\in L^1(\R)$ such that there is $x_0\in \R$ that is a Lebesgue point of $f$ with $f(x_0)\geq c$, and a Lebesgue point $y_0>x_0$ with $f(y_0)\leq -c$ for some $c>0$. Let  
\[
f_e(x)=\sum_{k=-1}^1 f(x+k(y_0-x_0))\chi_{[x_0,y_0[}(x+k(y_0-x_0)), \quad x\in \mathbb R,
\]
so that $f_e=f$ on $[x_0,y_0[$ and is a $(y_0-x_0)$-periodic function on $[x_0-(y_0-x_0),y_0+(y_0-x_0)[$, and zero elsewhere. Then there exists $g\in L^\infty([x_0,y_0])$ such that almost everywhere on $[x_0,y_0]$, it holds that 
\[
Hf \geq Hf_e  - g.
\]
\end{lem}
\begin{proof}
If $h$ is a function, we denote $h_x(t) = h(t+x)$.

Let $\delta_0 \in]0,(y_0-x_0)/2[$ be small enough to ensure 
\begin{equation}\label{eq:DensityAppl1}
\frac{1}{ h } \int_{x_0-h}^{x_0} f(t)  \, dt \geq \frac{c}{2},\quad \frac{1}{ h } \int_{y_0}^{y_0 + h} f_e(t)  \, dt \geq \frac{c}{2},
\end{equation}
\begin{equation}\label{eq:DensityAppl2}
\frac{1}{ h } \int_{y_0}^{y_0 + h} (-f(t)) \, dt \geq \frac{c}{2} \quad \text{and} \quad \frac{1}{ h } \int_{x_0-h}^{x_0} (-f_e(t) ) \, dt \geq \frac{c}{2}
\end{equation}
for every $0<h<2\delta_0$.

For $x \in ]x_0,x_0+\delta_0[$ and $0<\epsilon < x-x_0$, we have

\begin{align*}
\frac{1}{2}H_{\epsilon}^{(\delta_0)} M(f_x - f_{e,x}) &= \int_{\epsilon}^{\delta_0} \frac{1}{t}  \int_0^t \big(  (f_x(s) - f_x(-s) )  -  (f_{e,x}(s) - f_{e,x}(-s) ) \big) \, ds \, \frac{dt}{t} \\
&= \int_{\epsilon}^{\delta_0} \frac{1}{t}  \int_0^t \big(  f_x(s)  -  f_{e,x}(s)  \big) \, ds \, \frac{dt}{t} - \int_{\epsilon}^{\delta_0} \frac{1}{t}  \int_0^t \big(  f_x(-s)  - f_{e,x}(-s) \big) \, ds \, \frac{dt}{t} \\
&=- \int_{\epsilon}^{\delta_0} \frac{1}{t}  \int_{-x}^{t-x} \big(  f(-s)  - f_e(-s) \big) \, ds \, \frac{dt}{t} \\
&= -\int_{x-x_0}^{\delta_0} \frac{1}{t}  \int_{-x_0}^{t-x} \big(  f(-s)  - f_e(-s) \big) \, ds \, \frac{dt}{t} \\
&= \int_{x-x_0}^{\delta_0} \frac{t+x_0-x}{t} \Bigg(   \frac{1}{t+x_0-x}  \int_{x_0-(t+x_0-x)}^{x_0}   f_e(s) \, ds \\    
&\hspace{3cm}  - \frac{1}{t+x_0-x}  \int_{x_0-(t+x_0-x)}^{x_0}   f(s) \, ds   \Bigg) \, \frac{dt}{t}.
\end{align*}
Since $0<x-x_0<t<\delta_0$ implies $0<t+x_0-x<\delta_0$, using \eqref{eq:DensityAppl1} and \eqref{eq:DensityAppl2} we obtain
\begin{equation}\label{eq:HepsOnRightNeightborX}
\frac{1}{2} \limsup_{\epsilon\to 0^+}H_{\epsilon}^{(\delta_0)} M(f_x - f_{e,x}) \leq  \int_{x-x_0}^{\delta_0} \frac{t+x_0-x}{t^2}  \bigg(   \Big( - \frac{c}{2}\Big)  - \frac{c}{2}  \bigg)  \, dt  \leq  0 ,\quad  x \in ]x_0,x_0+\delta_0[.
\end{equation}

Similarly, for $x\in ]y_0 - \delta_0, y_0[$ and $0<\epsilon<y_0-x$, we have 
\begin{align*}
\frac{1}{2}H_{\epsilon}^{(\delta_0)} M(f_x - f_{e,x}) &= \int_{\epsilon}^{\delta_0} \frac{1}{t}  \int_0^t \big(  f_x(s)  -  f_{e,x}(s)  \big) \, ds \, \frac{dt}{t} - \int_{\epsilon}^{\delta_0} \frac{1}{t}  \int_0^t \big(  f_x(-s)  - f_{e,x}(-s) \big) \, ds \, \frac{dt}{t} \\
&= \int_{\epsilon}^{\delta_0} \frac{1}{t}  \int_0^t \big(  f_x(s)  -  f_{e,x}(s)  \big) \, ds \, \frac{dt}{t}\\
&= \int_{y_0 - x}^{\delta_0} \frac{1}{t}  \int_{y_0}^{t+x} \big(  f_x(s-x)  -  f_{e,x}(s-x)  \big) \, ds \, \frac{dt}{t} \\
&= \int_{y_0 - x}^{\delta_0}  \frac{t+x-y_0}{t} \Bigg(    \frac{1}{t+x-y_0}  \int_{y_0}^{y_0+(t+x-y_0)}   f(s) \, ds \\
&\hspace{3cm}-  \frac{1}{t+x-y_0}  \int_{y_0}^{y_0+(t+x-y_0)} f_e(s) \, ds  \Bigg) \, \frac{dt}{t}.
\end{align*}
and since $0<t+x-y_0<\delta_0$, we obtain

\begin{equation}\label{eq:HepsOnLeftNeightborY}
\frac{1}{2} \limsup_{\epsilon\to 0^+}H_{\epsilon}^{(\delta_0)} M(f_x - f_{e,x}) \leq  \int_{y_0 - x}^{\delta_0} \frac{t+x-y_0}{t^2}  \bigg(  \Big( - \frac{c}{2}\Big) - \frac{c}{2}    \bigg)  \, dt  \leq  0 ,\quad  x \in ]y_0-\delta_0,y_0[.
\end{equation}
Combining \eqref{eq:HepsOnRightNeightborX} and \eqref{eq:HepsOnLeftNeightborY} yields that for $x\in N:=]x_0,x_0+\delta_0[\cup ]y_0-\delta_0,y_0[$,
\[
\frac{1}{2} \limsup_{\epsilon\to 0^+}H_{\epsilon}^{(\delta_0)} M(f_x - f_{e,x}) \leq 0
\]
holds. By \eqref{eq:eqForDeltaHilbertTransform}, using the fact that $x\in]x_0,y_0[$ implies $Mf_{e,x}(\epsilon) = Mf_x(\epsilon) $ when $\epsilon<\min\{x-x_0,y_0-x\}$, it follows that for almost every $x\in N$
\[
0 \geq \limsup_{\epsilon\to 0^+} H_{\epsilon}^{(\delta_0)}(f_x-f_{e,x}) - M(f_x-f_{e,x})(\delta_0) = \lim_{\epsilon\to 0^+} H_{\epsilon}^{(\delta_0)}f_x  -  \lim_{\epsilon\to 0^+} H_{\epsilon}^{(\delta_0)} f_{e,x} - M(f_x-f_{e,x})(\delta_0),
\]
since the limits exists and are real numbers a.e. $x\in N$. 

Finally, since 
\[
\sup_{h\in B_{L^1(\R)}} \esssup_{x\in\R} \abs{ \pi H h(x) + \lim_{\epsilon\to 0^+} H_{\epsilon}^{(\delta_0)}h_x } \leq \frac{1}{\delta_0},
\]
it follows that for almost every $x\in N$ (notice that $3\norm{f}_{L^1(\mathbb R)} \geq \norm{f_e}_{L^1(\mathbb R)}$)
\[
\pi Hf_e(x) \leq  \frac{4\norm{f}_{L^1(\mathbb R)}}{\delta_0} +  \pi Hf(x) +   \sup_{x\in N} M(f_x-f_{e,x})(\delta_0),
\]
proving the statement of the lemma when $x\in N$. We conclude the proof by noting that since $f=f_e$ on $]x_0,y_0[$, it holds that $H(f_e-f)(x)$ is bounded for $x\in ]x_0,y_0[\setminus N$.

\end{proof}

We are now ready to prove the main local integrability result for real valued functions.

\begin{thm}\label{thm:local}
Let $Y\colon\R\to[0,\infty[$ be convex, and suppose that
$$
    \int_0^\infty Y(x)e^{-x}\,dx=\infty.
$$
Let $f\in L^1([0,2\pi[)$ and $I\subset ]0,2\pi[$ be a closed interval. Assume there are Lebesgue points of $f$, $x_0,y_0\in I$ with $x_0<y_0$ such that $f(x_0)\geq \pi/2$ and $f(y_0)\leq -\pi/2$, and $\abs{I\cap f^{-1}(]-\pi/2,\pi/2[)}=0$. Then 
\begin{equation}\label{eq:nonExpInt}
Y\big(\tilde{f} \big) \notin L^1(I).
\end{equation}

\end{thm}

\begin{proof}
Since  $\int_0^\infty Y(x)e^{-x} \, dx = \infty$, $Y$ is unbounded, so we can choose $b>0$ such that $Y(b)>Y(0)$. By convexity, for $b\le y<x$,
$$
    \frac{Y(x)-Y(y)}{x-y} \ge \frac{Y(b)-Y(0)}b >0,
$$
which shows that $Y$ is increasing on $[b,\infty[$. Define
$$
    Y_b(x) = \begin{cases}
        0, & x\le b,\\
        Y(x)-Y(b), & x>b.
    \end{cases}
$$
Then $Y_b$ is convex, increasing, $0\leq Y_b\leq Y$, and 
\begin{equation}\label{eq:Yb_Unbdd}
    \int_0^\infty Y_b(x)e^{-x}\,dx
    =\int_b^\infty \bigl(Y(x)-Y(b)\bigr)e^{-x}\,dx
    =\infty.
\end{equation}

 We begin by noting that it is sufficient to prove that
\[
Y_b\big(\tilde{f} \big) \notin L^1(]x_0,y_0[).
\]
Since $Y_b$ is increasing, by Lemma \ref{HilbertTransformAndConjugate}, there exists a constant $m_1>0$ such that

\begin{equation}\label{eq:HilbertTransformNotIntegrable1}
\int_{x_0}^{y_0}  Y_b\big(\tilde{f} \big)(x)  \,dm(x) \geq \int_{x_0}^{y_0} Y_b\big(Hf_e(x) - m_1 \big) \,dm(x),
\end{equation}
where $f_e$ is the semi-periodic extension given in the lemma. By Lemma \ref{HilbertTransformAndTranslations},we obtain
\begin{equation}\label{eq:ApplLemLinear}
\int_{x_0}^{y_0} Y_b\big(Hf_e(x) - m_1 \big) \,dm(x) = \frac{y_0-x_0}{2\pi} \int_0^{2\pi} Y_b\big(H(f_e\circ \psi)(x) - m_1 \big) \, dm(x),
\end{equation}
where $\psi(x) = \frac{y_0-x_0}{2\pi} x + x_0$.

Since $0$ and $2\pi$ are Lebesgue points of $f_e\circ \psi$ with $(f_e\circ \psi)(0)\geq \pi/2$ and $(f_e\circ \psi)(2\pi)\leq -\pi/2$, an application of Lemma \ref{lem:mainLem} yields (notice that $(f_e\circ \psi)_e = (f\circ \psi)_e$ on $[0,2\pi[$)
\begin{equation}\label{eq:ApplLem}
\int_0^{2\pi} Y_b\big(H(f_e\circ \psi)(x) - m_1 \big) \, dm(x)\geq \int_0^{2\pi} Y_b\big(H(f\circ \psi)_e(x) - m_2 \big) \, dm(x)  
\end{equation}
for some $m_2 \geq m_1$. Moreover, due to the $2\pi$-periodicity of $(f\circ\psi)_e $ on $[-2\pi,4\pi[$, another application of  Lemma \ref{HilbertTransformAndConjugate} yields
\[
\int_0^{2\pi} Y_b\big(H(f\circ \psi)_e(x) - m_2 \big) \, dm(x) \geq \int_0^{2\pi} Y_b\big(\widetilde{(f\circ\psi)_e }(x) - m_3 \big) \, dm(x)
\]
for some $m_3>m_2$. Combining this with \eqref{eq:HilbertTransformNotIntegrable1}, \eqref{eq:ApplLemLinear} and \eqref{eq:ApplLem}, we have obtained
\[
\int_{x_0}^{y_0}  Y_b\big(\tilde{f} \big)(x)  \,dm(x) \geq \frac{y_0-x_0}{2\pi}\int_0^{2\pi} Y_b\big(\widetilde{(f\circ\psi)_e }(x) - m_3 \big) \, dm(x)
\]
for some $m_3>0$. Since $(f\circ\psi)_e|_{[0,2\pi[}^{-1}([\pi/2,\infty[)$ is neither $[0,2\pi[$ nor 
empty modulo null sets, 
$(f\circ\psi)_e\in L^1([0,2\pi[)$, and $\abs{[0,2\pi[ \cap  (f\circ\psi)_e^{-1}(]-\pi/2,\pi/2[)} = 0$, we can apply Theorem \ref{thm:global} using $x\mapsto Y_b(x - m_3)$ (since the conditions on $Y_b$ are translational invariant) to obtain that
\begin{align*}
\int_I Y\big(\tilde{f} \big)(x)  \,dm(x)  &\geq \int_{x_0}^{y_0}  Y_b\big(\tilde{f} \big)(x)  \,dm(x)  \geq  \frac{y_0-x_0}{2\pi}\int_0^{2\pi} Y_b\big(\widetilde{(f\circ\psi)_e }(x) - m_3 \big) \, dm(x) \\
&\gtrsim_{f,\psi} \int_{\R}Y_b(\omega-m_3)e^{-\abs{\omega}}\,d\omega  =  e^{-m_3}  \int_0^\infty Y_b(x)e^{-x}\,dx,
\end{align*}
where we have used the fact that  $Y\geq Y_b$ on $\mathbb R$. By \eqref{eq:Yb_Unbdd}, we can conclude that $Y\big(\tilde{f} \big) \notin L^1(I)$.

\end{proof}

Under the gap assumption in Theorem~\ref{thm:global}, the condition
$$
    0<|f^{-1}([\pi/2,\infty[)|<2\pi
$$
is equivalent to the existence of Lebesgue points $x_0,y_0\in\T$ such that
$$
    f(x_0)\ge \pi/2\quad\text{and}\quad f(y_0)\le-\pi/2.
$$
On the circle, by choosing an appropriate angular parametrization, we can ensure the $\pi$-gap is positively oriented, while on an interval the ordering $x_0<y_0$ is essential for the positive-tail nonintegrability in the preceding theorem.

\section{About the complex case}\label{sec:complex}

The main result of this paper can be generalized to complex-valued functions in the following way: Let $I\subset [0,2\pi[$ be an interval and let $S'$ be an open strip in $\C$ of width $\pi$. If $f\colon I\mapsto \C\setminus S'$ belongs to $L_{\C}^1(\T)$ and is such that the inverse image of each component of $\C\setminus S'$ has positive measure, then $\textrm{exp}(\abs{\smash{\tilde{f}}}) $ is not integrable on $I$. 

In order to rigorously state a more general result, we define the following partition of $\C$:

\[
C_{\leq -\frac{\pi}{2}} := \{ z\in\C : \Re z \leq -\pi/2 \}, \quad  C_{\geq \frac{\pi}{2}} := \{ z\in\C : \Re z \geq \pi/2 \},\text{ and } \quad S = \{ z\in\C : \abs{\Re z}< \pi/2 \}.
\]

\begin{thm}\label{thm:complex}
Let $Y:\R\to[0,\infty[$ be convex, and suppose that
$$
    \int_0^\infty Y(x)e^{-x}\,dx=\infty.
$$
 Let $I\subset [0,2\pi[$ be an interval and let $f\in L_{\C}^1([0,2\pi[)$ and assume we can choose $c\in\C,\, \theta \in \R$ such that there exist density points $x_0$ and $y_0>x_0$ of $f|_I^{-1} \big( c+e^{i\theta} C_{\geq \frac{\pi}{2}}\big)$ and $f|_I^{-1} \big(c+e^{i\theta}C_{\leq -\frac{\pi}{2}}\big)$, respectively. If $\abs{I\cap f^{-1}(c+e^{i\theta} S)} = 0$, then $Y\bigl(\Re(e^{-i\theta} \tilde{f})\bigr) \notin L^1(I)$, in particular, $\exp(e^{-i\theta} \tilde{f}) \notin L_{\C}^1(I)$.
\end{thm}
\begin{proof}
Consider the function $g:= e^{-i\theta} (f - c)$, for which $m(I \cap g^{-1}(S))=0$. According to the assumption, there exist density points $x_0$ and $y_0>x_0$ of 
\[
g|_I^{-1}\big(  C_{\geq \frac{\pi}{2}} \big) = f|_I^{-1} \big( c+e^{i\theta} C_{\geq \frac{\pi}{2}}\big)
\]
and
\[
g|_I^{-1}\big(  C_{\leq -\frac{\pi}{2}} \big) = f|_I^{-1} \big( c+e^{i\theta} C_{\leq -\frac{\pi}{2}}\big),
\]
respectively. Moreover, since they are density points, we can find disjoint closed intervals $ I_{x_0}\ni x_0$ and $I_{y_0}\ni y_0$ such that 
\begin{equation}\label{eq:densityToLebesguePoints}
I_{x_0}\cap {g|_I}^{-1}\big(  C_{\geq \frac{\pi}{2}} \big) \quad\text{ and }\quad  I_{y_0}  \cap {g|_I}^{-1}\big(  C_{\leq -\frac{\pi}{2}} \big)
\end{equation}
have positive measures. Therefore, we can find Lebesgue points $x_0'$ and $ y_0'$ of $\Re g$ from the respective intersections given in \eqref{eq:densityToLebesguePoints}. It follows that using $\Re g$, $I'=[x_0',y_0']$, where $x_0'<y_0'$, we can apply Theorem \ref{thm:local} to obtain
\[
\int_I Y\bigl(  \Re (e^{-i\theta} \tilde{f}(x)) \bigr)  \,  dx =\int_I Y\bigl( \Re (\COCO{e^{-i\theta}  (f - c)}(x)  ) \bigr)   \,  dx = \int_I Y\bigl(\widetilde{\Re g}(x) \bigr) \,  dx = \infty.
\]

Concerning the exponential function, the statement follows from $\abs{e^z} = e^{\Re z}$.

\end{proof}

In particular, we have:
\begin{cor}
 Let $I\subset [0,2\pi[$ be an interval and let $f\in  L^1_{\C}([0,2\pi[)$ and assume we can choose $c\in\C, \theta \in \R$ such that both $f|_I^{-1} \big( c+e^{i\theta} C_{\geq \frac{\pi}{2}}\big)$ and $f|_I^{-1} \big(c+e^{i\theta}C_{\leq -\frac{\pi}{2}}\big)$, which are subsets of $I$, have positive Lebesgue measure.  If $\abs{I\cap f^{-1}(c+e^{i\theta} S)} = 0$, then 
\[
\exp(|\tilde{f}|) \notin L^1(I).
\]
\end{cor}

\section{Sufficient condition for  \texorpdfstring{$Y$}{hm}-integrability}\label{sec:SuffCondAndPWgaps}

We begin the section with a counterpart to Lemma \ref{lem:strip}.

\begin{lem}\label{lem:FYbounded}
Let $Y\colon\R\to[0,\infty[$ be convex.
Set $\Phi_Y(z)=Y(\Im z)$. And let
$$
    F(z)=-i\log\frac{1+z}{1-z},\qquad z\in\D,
$$
where the logarithm is the branch on the right half-plane that is real on $]0,\infty[$. For every $b\in\R$, the function $F_b=F+ib$ satisfies
\[ 
    \int_{\T}\Phi_Y(F^*_b(\zeta))\,dm(\zeta)\leq \frac{  2e^{|b|} }{\pi}\int_{\R} Y(x)e^{-\abs{x}}\,dx.
\]
\end{lem}

\begin{proof}
First, we have
\begin{align*}
\int_{\R}\frac{Y(\omega)}{\cosh(\omega-b)}\,d\omega &\leq e^{|b|} \int_{\R}\frac{Y(\omega)}{\cosh \omega}\,d\omega \leq  2e^{|b|}  \int_{\R}  Y(\omega) e^{-\abs{\omega}} \,d\omega,
\end{align*}
where the first inequality follows from $\cosh(\omega-b)\geq e^{-|b|}\cosh \omega$ and the second from $\cosh \omega\geq e^{\abs{\omega}}/2$, $\omega\in\R$.
Using the calculations done in Lemma \ref{lem:strip}, we can conclude that
\begin{equation*}
\int_{\T}\Phi_Y(F^*_b(\zeta))\,dm(\zeta) \leq \frac{  2e^{|b|} }{\pi}\int_{\R} Y(x)e^{-\abs{x}}\,dx.
\end{equation*}

\end{proof}

We are now ready to present the main result of this section, which is a complement to Theorem \ref{thm:global}.

\begin{thm}\label{thm:finite}
Let $Y\colon\R\to[0,\infty[$ be convex, and satisfy the growth condition
\begin{equation}\label{e:Y-critical-growthR}
    \int_{\R} Y(x)e^{-\abs{x}}\,dx<\infty.
\end{equation}
If $g\in L^\infty(\T)$ with $\|g\|_\infty\leq \frac{\pi}{2}$, then 
\[
Y( \tilde g) \in L^1(\T). 
\]
\end{thm}

\begin{proof}
If $g$ is constant, $\tilde g \equiv 0$, so we may assume $g$ is not constant. Similarly to the proof of Theorem \ref{thm:global},
set $\Phi_Y(z)=Y(\Im z)$, which is subharmonic, and define
$$
    G(z)=\int_{\T}\frac{\xi+z}{\xi-z}\,g(\xi)\,dm(\xi),\qquad z\in\D.
$$
Then $\Re G=P[g]$, $\Im G(0)=0$, and $G^*=g+i\widetilde g$ a.e. on $\T$. Thus, by the maximum modulus principle
\begin{equation}\label{eq:Gstays}
    G(z)\in S\qquad\text{for }z\in \D,
\end{equation}
where $S$ is defined in \eqref{eq:vertStrip}. We can, therefore, find $b\in\R$ such that $G(0) = b$. Moreover, let
\[
    F(z)=-i\log\frac{1+z}{1-z},\qquad z\in\D.
\]
It holds that $F(0)=0$ and
\begin{equation}\label{eq:Fleaves}
    \abs{\Re F^*(\zeta)}\geq \frac{\pi}{2}\qquad\text{for a.e.\ }\zeta\in\T,
\end{equation}
since $F$ maps $\D$ conformally onto $S$ (Lemma \ref{lem:strip}). Since $F$ is surjective, there is an automorphism $\sigma$ of $\D$ with $(F\circ\sigma)(0) = b = G(0)$. It is clear that \eqref{eq:Fleaves} still holds for $F_\sigma:=F\circ\sigma$.

Using $\Omega = S$, we can apply Theorem \ref{thm:burkholder} to obtain (the last inequality)
\begin{equation}\label{eq:ApplOfBurkholder}
\int_{\T} Y(\tilde g)\,dm(\zeta) \leq    \sup_{0<r<1}\int_{\T}\Phi_Y(G(r\zeta))\,dm(\zeta)
\leq\sup_{0<r<1}\int_{\T}\Phi_Y(F_\sigma(r\zeta))\,dm(\zeta),
\end{equation}
where the first inequality is due to Fatou's lemma and the definitions of $\Phi_Y$ and $G$. Since $F_{\sigma}\in H^1$, its imaginary part is the Poisson integral of $\Im F_{\sigma}^*=\Im(F^*\circ\sigma)$, and so Lemma~\ref{lem:supEqualsBdd} implies that
\begin{equation}\label{eq:YoungsIneqBdd}
\sup_{0<r<1}\int_{\T}\Phi_Y(F_\sigma(r\zeta))\,dm(\zeta) = \norm{Y\circ \Im F^*\circ\sigma}_{L^1} \leq \norm{\sigma'}_{L^\infty} \norm{Y\circ \Im F^*}_{L^1}
\end{equation}
by a change of variables. Combining \eqref{eq:ApplOfBurkholder}, \eqref{eq:YoungsIneqBdd}, and Lemma~\ref{lem:FYbounded}, we have
\[
\int_{\T} Y(\tilde g)\,dm(\zeta) \leq  \norm{\sigma'}_{L^\infty} \norm{Y\circ \Im F^*}_{L^1} \leq   \norm{\sigma'}_{L^\infty}  \frac{  2 }{\pi} \int_{\R} Y(x)e^{-\abs{x}}\,dx.
\]

\end{proof}

We obtain the following two results from Theorems \ref{thm:global} and \ref{thm:finite}, and their proofs. 

\begin{cor}\label{cor:CharacterizationOfIntegrability}
Let $Y\colon\R\to[0,\infty[$ be convex. If $E\subset\T$ is measurable with $0<|E|<2\pi$, then
 \[
 Y(\pi\widetilde{\chi_E})\in L^1(\T)
 \]
 if and only if
 \[
 \int_{\R} Y(\omega)e^{-\abs{\omega}}\,d\omega<\infty.
 \]
Moreover,
\[
\norm{Y(\widetilde{\chi_E})}_{L^1(\T)}\asymp_E \int_{\R} Y\bigg(\frac{\omega}{\pi}\bigg)e^{-\abs{\omega}}\,d\omega \asymp \int_{\R} Y(\omega)e^{-\pi\abs{\omega}}\,d\omega .
\]
\end{cor}

\begin{cor}\label{cor-strip-estimate}
Let $Y\colon\R\to[0,\infty[$ be convex, $\zeta_0\in\C$, and
\[
    F_0(z)=-i\log\frac{1+z}{1-z} + \zeta_0,\qquad z\in\D.
\]
 Then
 \[
\frac{e^{-|\Im \zeta_0|}}{\pi}\int_{\R} Y(\omega)e^{-\abs{\omega}}\,d\omega \leq  \int_{\T}Y\bigl(\Im F_0^*(\zeta)\bigr)\,dm(\zeta)\leq \frac{2 e^{|\Im \zeta_0|}}{\pi}\int_{\R} Y(\omega)e^{-\abs{\omega}}\,d\omega.
 \]
\end{cor}

The first one, Corollary \ref{cor:CharacterizationOfIntegrability}, gives us a rather precise class to where the conjugate functions of indicator functions belong. 

Using the proof of Theorem \ref{thm:local}, we have for increasing functions $Y$ the following, local version of Corollary \ref{cor:CharacterizationOfIntegrability}.

\begin{cor}
Let $Y\colon\R\to[0,\infty[$ be convex and increasing. Let  $I\subset ]0,2\pi[$ be a closed interval with $\abs{I}<2\pi$, and $E\subsetneq I$. Assume there are $x_0,y_0\in I$, with $x_0<y_0$ that are density points of $E$ and $]0,2\pi[\setminus E$, respectively. Then
\[
\norm{Y(\widetilde{\chi_E})}_{L^1(I)}\asymp_{E,I} \int_{\R} Y(\omega)e^{-\pi\abs{\omega}}\,d\omega \asymp \int_{[0,\infty[} Y(\omega)e^{-\pi\abs{\omega}}\,d\omega.
\]
\end{cor}
\begin{proof}
The second $\asymp$ follows from $Y\geq 0$ being increasing. Following the proof of Theorem \ref{thm:complex}, we can assume $x_0$ and $y_0$ are Lebesgue points of $\chi_E$. Put $Y_1(x)= Y(x/\pi)$, which is increasing and convex, and define $f=\tfrac\pi2\rho_E$. Apply the proof of Theorem \ref{thm:local} (where we need $Y$ to be increasing) to obtain
\[
\int_{x_0}^{y_0}  Y\big(\widetilde{\chi_E}(x) \big)  \,dm(x) = \int_{x_0}^{y_0}  Y_1\big(\tilde{f}(x) \big) \,dm(x) \geq \frac{y_0-x_0}{2\pi}\int_0^{2\pi} Y_1\big(\widetilde{(f\circ\psi)_e }(x) - m_0 \big) \, dm(x)
\]
for some $m_0>0$. By Theorem \ref{thm:global}, it follows that
\[
\int_I Y\big(\widetilde{\chi_E}(x) \big)  \,dm(x) \geq \int_{x_0}^{y_0}  Y\big(\widetilde{\chi_E}(x) \big)  \,dm(x)  \gtrsim  \int_{\R} Y_1(\omega - m_0) e^{-\abs{\omega}}\, d\omega \asymp \int_{\R} Y(\omega) e^{-\pi\abs{\omega}}\, d\omega. 
\]
From the assumption, it is clear that $0<|E|<2\pi$, so we can obtain the upper bound from Corollary \ref{cor:CharacterizationOfIntegrability}.
\end{proof}

\begin{exam}\label{ex:IllustrativeExampleLocal}
Let $E\subsetneq \T$ be an arc and consider the function $f = \tfrac\pi2\rho_E$. Traversing counterclockwise, let's say we enter $E$ at $s$ and exit $E$ at $t$. Since the conjugate function is a convolution with an odd kernel, $\tilde f$ tends to $-\infty$, when approaching $s$, while it tends to $\infty$ when approach $t$. Let $I\subset \T$ be a small enough neighborhood of $t$ to grant $\tilde f$ is positive on $I$. Let $Y\colon\R\to[0,\infty[$ be convex. On the one hand, if   
\[
    \int_0^\infty Y(x)e^{-x}\,dx = \infty,
\]
then $Y(\tilde f)\notin L^1(I)$ by Theorem \ref{thm:local}. On the other hand, if
\[
    \int_0^\infty Y(x)e^{-x}\,dx < \infty,
\]
then we can extend $Y|_{[0,\infty[}$ linearly to $\R$ and keep the convexity; we denote the extension by $Y_0$. Now
\[
    \int_{\R} Y_0(x)e^{-|x|}\,dx < \infty,
\]
and Corollary \ref{cor:CharacterizationOfIntegrability} yields
\[
Y(\tilde f) = Y_0(\tilde f) \in L^1(I),
\]
where the equality is due to $\tilde f$ being positive on $I$.
\end{exam}

Since $x\mapsto (e^x/(1+x)-1)\chi_{[0,\infty[}(x)$ is increasing and convex on $\R$, see Example  \ref{ex:BetterThanExp}, we present the following example.

\begin{exam}\label{ex:BetterThanExpLocal}

Let $Y(x)=(e^x/(1+x)-1)\chi_{[0,\infty[}(x)$. Let $E\subset\T$ be measurable with $0<|E|<2\pi$, and $I\subset ]0,2\pi[$ be an interval and $\abs{I}<2\pi$. If $f\in L^1([0,2\pi[)$ with $f\geq \frac{\pi}{2}$ and there are points $x_0,y_0\in I$ with $x_0<y_0$, where $x_0$ is a density point of $E$ and $y_0$ is a density point of $I\setminus E$, then
\[
Y\big(\widetilde{\rho_E f} \big) \notin L^1(I).
\]

\end{exam}

Notice that the assumption on $E$ holds if $I\cap E$ or $I\setminus E$ is not an interval modulo null sets (compare with \cite[Lemma]{Shargorodsky-1994}).

So far we have considered gaps in the essential range, either on the whole circle or on an arc. 
Theorems \ref{thm:global}, \ref{thm:local} and \ref{thm:finite} give a good picture of what information we can obtain by examining gaps. However, we can increase the precision of the research on integrability of the conjugate functions by examining pointwise gaps. For an open arc $I$, a function $f\colon I\to \R$ has a pointwise gap at $c\in I$ if $f|_{[c-\epsilon,c+\epsilon]}$ has a gap for some arbitrarily small $\epsilon>0$. The size of the pointwise gap is given by
\begin{equation}\label{eq:pointwiseGap}
\limsup_{\epsilon\to 0^+} \,  \sup\{\delta>0 : f \text{ has a } \delta \text{-gap on } [c-\epsilon,c+\epsilon]\cap I\},
\end{equation}
which might be infinite. 

To initiate further research on this matter, we turn to the simplest interesting case, that is, functions with jump discontinuities at which no gap of the full jump size exists. The following example shows that a pointwise gap of size $\pi$ does not, in general, force nonintegrability. A different example of this phenomenon is given in \cite[Example 2]{Gissy-2021}.

\begin{exam}\label{exam:jumpDiscontButExpIntegrable}
Consider a function $f\colon \R\to [0,\infty[$ with support on $]1/2,1[$, its restriction to $]0,1[$ is in $C^\infty$, and on $]2/3,1[$ it is defined by 
\[
t\mapsto \pi \Big(1-2/\log \frac{e}{1-t}\Big).
\]
Then, although $\lim_{x\to 1^-} f(x) - \lim_{x\to 1^+} f(x) = \pi - 0 = \pi$, we have $\exp(Hf) \in L_{\textrm{loc}}^1(\R)$. To help the reader verify that $\exp(Hf)\in  L^1(]8/9,10/9[)$, we will sketch a proof. 
Assume $x\in]8/9,10/9[\setminus \{1\}$. Instead of $Hf$, it is enough to consider
\[
\PV_{2/3}^1 \frac{g(t)\, dt}{x-t} = \PV_{2/3}^1 \frac{g(t) - 1\, dt}{x-t} +  \log \frac{x-2/3}{\abs{x-1}},
\]
where $g=f/\pi$. Although, the second term lacks exponential integrability, the function in question makes the first term compensate. 
 We have
\begin{align*}
\PV_{2/3}^1 \frac{g(t) - 1\, dt}{x-t} &= -2\PV_{2/3}^1 \frac{ dt }{(x-t)\log\frac{e}{1-t}} = -2\PV_0^{1/3} \frac{ dt }{(x-1+t)\log\frac{e}{t}} \\
&= -2\PV_0^{\frac{1}{3\abs{x-1}}} \frac{ \abs{x-1} dt }{(x-1+t\abs{x-1})\log\frac{e}{t\abs{x-1}}} \\
&= -2\bigg(\PV_0^e+\int_e^{\frac{1}{3\abs{x-1}}} \bigg) \frac{ dt }{(\frac{x-1}{ \abs{x-1} }+t)\log\frac{e}{t\abs{x-1}}}.
\end{align*}
Next, we prove that the first integral is bounded. The case  $x>1$ follows immediately since the integrand is uniformly bounded w.r.t. $x\in]8/9,10/9[$. If $x<1$ we add and subtract $(\log (e/(1-x)))^{-1} $ to obtain
\[
 \PV_0^e  \frac{ dt }{( t -1 )\log\frac{e}{t\abs{x-1}}} =\PV_0^e \Bigg(    \frac{ 1 }{\log\frac{e}{t\abs{x-1}}} -  \frac{ 1 }{\log\frac{e}{\abs{x-1}}}  \Bigg)  \frac{ dt }{( t -1 )}  + \frac{ 1 }{\log\frac{e}{\abs{x-1}}}  \PV_0^e  \frac{ dt }{( t -1 )}.
\]
Furthermore,
\begin{align*}
\PV_0^e \Bigg(    \frac{ 1 }{\log\frac{e}{t\abs{x-1}}} -  \frac{ 1 }{\log\frac{e}{\abs{x-1}}}  \Bigg)  \frac{ dt }{( t -1 )} &=  \frac{ 1 }{\log\frac{e}{\abs{x-1}}}  \PV_0^e  \frac{ \log t }{\log\frac{e}{t\abs{x-1}}}    \frac{ dt }{( t -1 )}. 
\end{align*}
Because $t\mapsto (\log t)/(t-1)$ is positive, the integrand is positive and clearly uniformly bounded w.r.t. $x\in]8/9,1[$. Since the principal value integral
\[
\PV_0^e  \frac{ dt }{( t -1 )} = \log (e-1)
\]
and 
\[
\sup_{x\in]8/9,10/9[}  \frac{ 1 }{\log\frac{e}{\abs{x-1}}}  <\infty, 
\]
we can conclude that 
\[
\sup_{x\in]8/9,10/9[\setminus\{1\}}\PV_0^e \frac{ dt }{(\frac{x-1}{ \abs{x-1} }+t)\log\frac{e}{t\abs{x-1}}}<\infty. 
\]

Furthermore, we have
\begin{align*}
\int_e^{\frac{1}{3\abs{x-1}}}  \frac{  t   }{   (\frac{x-1}{ \abs{x-1} }+t)  } \frac{ dt }{  t\log\frac{e}{t\abs{x-1}}} &\geq  \frac{  e   }{   1+e  }  \int_e^{\frac{1}{3\abs{x-1}}} \frac{ dt }{  t\log\frac{e}{t\abs{x-1}}} = \frac{  e   }{   1+e  }  \int_{e\abs{x-1}}^{\frac{1}{3}} \frac{ dt }{  t\log\frac{e}{t}} \\
&= \frac{  e   }{   1+e  } \bigg(\log\log\frac{1}{ \abs{x-1} }  - \log\log(3e) \bigg). 
\end{align*}

Combining everything we have done so far, we get for some $C>0$
\[
\PV_{2/3}^1 \frac{g(t)\, dt}{x-t} \leq  \log \frac{1}{\abs{x-1}} - 2\frac{  e   }{   1+e  } \log\log\frac{1}{ \abs{x-1} }  + C,
\]
hence,
\[
\exp \bigg( \PV_{2/3}^1 \frac{g(t)\, dt}{x-t}\bigg) \leq \frac{1}{\abs{x-1}} \bigg(\log\frac{1}{ \abs{x-1} } \bigg)^{- 2\frac{  e   }{   1+e  }} e^{C}. 
\]
Since $ 2e/ (1+e)>1$, it follows that 
\[
\int_{8/9}^{10/9}  \exp \bigg( \frac{1}{\pi}\PV_{2/3}^1 \frac{f(t)\, dt}{x-t}\bigg) \, dx <\infty,  
\]
and since $1$ is the only singularity of $Hf$ and $f$ has compact support $\exp(Hf) \in L_{\textrm{loc}}^1(\R)$.
\end{exam}

For piecewise H\"older continuous functions, the classical logarithmic asymptotics of Cauchy integrals yield the following integrability criterion; see, e.g.,~\cite{Musk}.

\begin{prop}\label{PC-prop}
Let $f\colon\T\to\R$ be piecewise continuous with jumps at $t_1, \ldots, t_N$, and suppose that each restriction $f|_{]t_j,t_{j+1}[}$ extends to a H\"older continuous function on $[t_j, t_{j+1}]$. Let
$$
    L=\max_{1\le k\le N} |J_k|,\qquad
    J_k=f(t_k^-)-f(t_k^+).
$$
If $Y\colon\R\to[0,\infty[$ is increasing, then
\begin{equation}\label{e:PC-prop}
    Y(|\tilde f|)\in L^1(\T) \iff
    \int_0^\infty Y(s)e^{-\pi s/L}ds<\infty.
\end{equation}
In particular, for $\lambda>0$,
$$
    \exp(\lambda|\tilde f|)\in L^1(\T) \iff \lambda L<\pi.
$$
\end{prop}
\begin{proof}
Notice that near each jump
$$
    \tilde f(x)=\frac{J_k}\pi\log\frac1{|x-t_k|}+O(1);
$$
see~\cite[3\textdegree\ in \S33 and (29.7) in \S29]{Musk}. Let $a_k=\frac1\pi|J_k|$. For a sufficiently small $\epsilon>0$, there is a $C_k$ such that
$$
    a_k\log\frac1{|x-t_k|}-C_k\le|\tilde f(x)|
    \le a_k\log\frac1{|x-t_k|}+C_k
$$
whenever $0<|x-t_k|<\epsilon$. Since $Y$ is increasing, for $b=\pm C_k$, we have
$$
    \int_0^\epsilon Y\Bigl(a_k\log\frac1r+b\Bigr)\, dr
    =\frac{e^{b/a_k}}{a_k}\int_{a_k\log\frac1\epsilon+b}^{\infty} Y(s)e^{-s/a_k}\, ds.
$$
Since $Y$ is bounded on compact intervals, the lower limit does not affect finiteness. Thus, integrability near $t_k$ is equivalent to
$$
    \int_0^\infty Y(s)e^{-\pi s/|J_k|}\, ds<\infty.
$$
Since $\tilde f$ is bounded away from the jumps and the largest jump gives the slowest decaying exponential weight, we obtain the equivalence in~\eqref{e:PC-prop}.
\end{proof}

\begin{rem}
(i) The assumption that $Y$ is increasing can be replaced by convexity. Indeed, if $Y$ is bounded on $[0,\infty[$, both conditions in $\eqref{e:PC-prop}$ hold. Otherwise, convexity implies that $Y$ is increasing on $[x_0,\infty[$ for some $x_0\ge 0$. Replacing $Y(x)$ by $Y(\max\{x,x_0\})$ reduces the claim to the proposition because this replacement does not change either integrability condition.

(ii) Let $f=\pi x\chi_{[0,1[}$ on $[0,2\pi[$. Then $f$ has a single jump, of size $\pi$ at $x=1$, and is Lipschitz on each side of it. For $0<h<1$, the essential range of $f$ on $]1-h,1+h[$ is $\{0\}\cup[\pi(1-h),\pi]$, so on each of these arcs $f$ has a gap of size $\pi(1-h)<\pi$, and on arcs avoiding the jump $f$ has no gap at all. Thus, $f$ does not have a $\pi$-gap on any arcs, and Theorem~\ref{thm:local} only yields $\exp(\lambda\tilde f)\notin L^1(\T)$ for $\lambda>1$. Nevertheless, Proposition~\ref{PC-prop} with $L=\pi$ gives $\exp(|\tilde f|)\notin L^1(\mathbb T)$. In contrast, the function in Example~\ref{exam:jumpDiscontButExpIntegrable}, viewed on $\T$, satisfies $\tilde f=Hf+O(1)$ near the jump, and its exponential is integrable there. Thus, when the local gap sizes approach $\pi$ without attaining it, the endpoint behavior can depend on how the one-sided limits are approached.
\end{rem}

\section*{Acknowledgements}
The first author has carried out the work at the University of Reading and Åbo Akademi University, funded by the Magnus Ehrnrooth Foundation, and at the Norwegian University of Science and Technology funded by the Ruth and Nils-Erik Stenbäck Foundation. The second author was supported in part by the Engineering and Physical Sciences Research Council grant EP/Y008375/1 and the Research Council of Finland grant 375631.

We would also like to thank Professor Eugene Shargorodsky for the fruitful discussions in the early stages of the project.

\bibliographystyle{abbrv}
\bibliography{bibliography}

\end{document}